\def\version{2 December  2011}

\documentclass[reqno,11pt]{amsart}

\usepackage{amssymb,epsf,color}
\usepackage{rotating}

\definecolor{gray}{rgb}{0.9,0.9,0.9}

\def\be{\begin{equation}}
\def\ee{\end{equation}}
\def\ba{\begin{align}}
\def\ea{\end{align}}
\def\bsplit{\begin{split}}
\def\esplit{\end{split}}
\def\bm{\begin{multline}}
\def\eem{\end{mutline}}
\def\bfig{\begin{figure}[htb]}
\def\efig{\end{figure}}

\numberwithin{equation}{section}
\newtheorem{theorem}{Theorem}[section]
\newtheorem{proposition}[theorem]{Proposition}
\newtheorem{lemma}[theorem]{Lemma}
\newtheorem{corollary}[theorem]{Corollary}

\newcommand{\nn}{\nonumber}

\renewcommand{\leq}{\;\leqslant\;}
\renewcommand{\geq}{\;\geqslant\;}
\newcommand{\dd}{{\rm d}}
\newcommand{\e}[1]{\,{\rm e}^{#1}\,}
\newcommand{\ii}{{\rm i}}
\newcommand{\sumtwo}[2]{\sum_{\substack{#1 \\ #2}}}

\def\Re{{\operatorname{Re\,}}}
\def\Im{{\operatorname{Im\,}}}

\newcommand{\upchi}{\raise 2pt \hbox{$\chi$}}

\newcommand{\caS}{{\mathcal S}}

\newcommand{\bbE}{{\mathbb E}}\newcommand{\bbN}{{\mathbb N}}\newcommand{\bbP}{{\mathbb P}}

\begin{document}

{\hfill\small \version} \vspace{2mm}

\title[Cycle structure of random permutations with cycle weights]{Cycle structure of random permutations \\ with cycle weights}

\author{Nicholas M.\ Ercolani, Daniel Ueltschi}

\address{Nicholas M.\ Ercolani \hfill\newline
\indent Department of Mathematics, \hfill\newline
\indent The University of Arizona \hfill\newline
\indent 617 N.\ Santa Rita Ave. \hfill\newline
\indent P.O.\ Box 210089 \hfill\newline
\indent Tucson, AZ 85721--0089 USA \hfill\newline
{\small\rm\indent http://math.arizona.edu/~ercolani/}
}
\email{ercolani@math.arizona.edu}

\address{Daniel Ueltschi \hfill\newline
\indent Department of Mathematics \hfill\newline
\indent University of Warwick \hfill\newline
\indent Coventry, CV4 7AL, United Kingdom \hfill\newline
{\small\rm\indent http://www.ueltschi.org}
}
\email{daniel@ueltschi.org}

\begin{abstract}
We investigate the typical cycle lengths, the total number of cycles, and the number of finite cycles in random permutations whose probability involves cycle weights. Typical cycle lengths and total number of cycles depend strongly on the parameters, while the distributions of finite cycles are usually independent Poisson random variables.
\end{abstract}

\maketitle


\noindent
{\footnotesize {\it 2010 Math.\ Subj.\ Class.:} 60C05.}

\tableofcontents

\section{Introduction}

Weighted random partitions and random permutations appear in mathematical biology and in theoretical physics. They are appealing because of their natural probabilistic structure and their combinatorial flavor. The sample space for random permutations is the set $\caS_n$ of permutations of $n$ elements. Given $\sigma \in \caS_n$, we let $R_j(\sigma)$ denote the number of cycles of length $j$ in $\sigma$. The probability of a permutation $\sigma$ is then defined as
\be
\label{basic prob}
\bbP_n(\sigma) = \frac1{n! h_n} \prod_{j\geq1} \theta_j^{R_j(\sigma)}.
\ee
Here, $\theta_1, \theta_2, \dots$ are nonnegative parameters and $h_n$ is the normalization that makes $\bbP_n$ a probability distribution, i.e.,
\be
h_n = \frac1{n!} \sum_{\sigma \in \caS_n} \prod_{j\geq1} \theta_j^{R_j(\sigma)}.
\ee
Notice that $R_1,R_2,\dots$ satisfy the following identity for all $\sigma \in \caS_n$:
\be
\label{sum rule}
\sum_{j=1}^n j R_j(\sigma) = n.
\ee

The case $\theta_j \equiv 1$ corresponds to random permutations with uniform distribution and it has been studied e.g.\ in \cite{DP,AT,BP}. The case $\theta_j \equiv \theta$ is known as the Ewens distribution and it was introduced for the study of population dynamics in mathematical biology \cite{Ewe}. See \cite{ABT,Han,FH} and references therein. Random permutations with restriction on the cycle lengths can be described with parameters $\theta_j \in \{0,1\}$. Results have been obtained for random permutations restricted to finite cycles \cite{Tim,Ben} or to cycle lengths of given parity \cite{Lugo}. Another situation of interest is when we fix the asymptotic behavior of $\theta_j$ for large $j$. Such a setting was considered in \cite{BG}, and it also appears in the study of the quantum Bose gas in statistical mechanics \cite{BU1,BU2}. The case where $\theta_j$ converges to a constant (i.e., the Ewens case, asymptotically) was considered in \cite{Lugo,BUV}. Vanishing parameters $\theta_j \to 0$ were studied in \cite{BG,BUV} where logarithmic cycle lengths were observed.

The random variables to be discussed in this article are the following:
\begin{itemize}
\item $L_1(\sigma)$ gives the length of the cycle that contains the index 1. Since our probability distribution is invariant under relabeling, we can interpret $L_1$ as giving the length of a ``typical cycle'', i.e.\ the length of the cycle that contains a random index.
\item $K(\sigma) = \sum_{j=1}^n R_j(\sigma)$ gives the total number of cycles in the permutation $\sigma$.
\item We also consider the distributions of $R_1,R_2,\dots$ as $n\to\infty$, i.e., the number of finite cycles. In most cases, they will be seen to converge to independent Poisson random variables.
\end{itemize}

The probability distribution $\bbP_n$ depends only on the cycle structure $(R_1(\sigma),\dots,R_n(\sigma))$ of $\sigma \in \caS_n$. It can therefore be understood as a distribution on nonnegative integers $(r_1,\dots,r_n)$ that satisfy $\sum j r_j = n$. As is well-known, these numbers are in one-to-one correspondence with integer partitions $(\lambda_1,\lambda_2,\dots)$ with $\lambda_1 \geq \lambda_2 \geq \dots$ and $\sum \lambda_i = n$, by defining $r_j = \#\{ i: \lambda_i = j \}$. Taking into account the number of partitions that correspond to a given set $(r_1,\dots,r_n)$, the probability of a weighted random partition is then
\be
\bbP_n(\lambda_1, \lambda_2, \dots) = \frac1{h_n} \prod_{j=1}^n \frac1{r_j!} \Bigl( \frac{\theta_j}j \Bigr)^{r_j}.
\ee
A random variable such as $L_1$ cannot be expressed in terms of $\{R_j\}$, but one can give an interpretation of its distribution in the context of random partitions. Given a random partition, pick $i \in \{1,\dots,n\}$ uniformly at random, and take $j$ that satisfies
\be
\lambda_1 + \dots + \lambda_{j-1} < i \leq \lambda_1 + \dots + \lambda_j.
\ee
The distribution of $\lambda_j$ is identical to that of $L_1$.

\begin{sidewaystable}
\centering
\vspace{16cm}
\colorbox{gray}{
\begin{tabular}{lcccc}
Parameters & Typical cycle lengths & Number of cycles & Finite cycles & Reference \\
\hline
$\theta_n = \e{n^\gamma}$, $\gamma>1$ & $\bbP_n(L_1 = n) \to 1$ & $K \Rightarrow 1$ & $R_j \Rightarrow 0$ & Section \ref{sec regime A} \\
\hline
$\theta_n \approx \e{n^\gamma}$, $0<\gamma<1$ & $L_1 / (\frac1{1-\gamma} \log n)^{1/\gamma} \Rightarrow 1$ &  & $R_j \Rightarrow \text{Poisson} \bigl( \frac{\theta_j}j \bigr)$ & Section \ref{sec regime B} \\
\hline
$\theta_n \approx n^\gamma$, $\gamma>0$ & \hspace{-3mm} $L_1 / n^{\frac1{\gamma \! + \! 1}} \Rightarrow \text{Gamma} \bigl( \gamma \! + \! 1, \Gamma(\gamma \! + \! 1)^{\frac1{\gamma \! + \! 1}} \bigr)$ & $\bbE_n(K) \approx A n^{\frac\gamma{\gamma+1}}$  & $R_j \Rightarrow \text{Poisson} \bigl( \frac{\theta_j}j \bigr)$ & Section \ref{sec regime C} \\
\hline
$\theta_n \to \theta$ & $L_1/n \Rightarrow \text{Beta}(1,\theta)$ & $\bbE_n(K) \approx \theta \log n$ & $R_j \Rightarrow \text{Poisson} \bigl( \frac{\theta_j}j \bigr)$ & Section \ref{sec regime D} \\
\hline
$\substack{\displaystyle \theta_n = n^{-\gamma}, \gamma>0, \text{ or} \\ \displaystyle \theta_n = \e{-n^\gamma}, 0<\gamma<1}$ & $L_1/n \Rightarrow 1$ & \hspace{-3mm} $K \Rightarrow 1 + \text{Poisson}(\sum \frac{\theta_j}j)$ & $R_j \Rightarrow \text{Poisson} \bigl( \frac{\theta_j}j \bigr)$ & Section \ref{sec regime E} \\
\hline
$\theta_n = \e{-n^\gamma}$, $\gamma>1$ & $L_1 / \bigl( \frac1{\gamma-1} \log n \bigr)^{1/\gamma} \Rightarrow 1$ &  & \hspace{-3mm} $\bbE_n(R_j) \sim \exp\bigl\{ j \gamma \bigl( \frac{\log n}{\gamma-1} \bigr)^{\frac{\gamma-1}\gamma} \bigr\}$ & Section \ref{sec regime F} \\
\end{tabular}
}
\medskip
\caption{Overview of typical cycle lengths ($L_1$), total number of cycles ($K$), and number of finite cycles ($R_j$), for different parameters. As for the notation: `$\Rightarrow$' denotes convergence in distribution; `$a_n \approx b_n$' means that $a_n/b_n \to 1$ as $n\to\infty$; `$a_n \sim b_n$' is a much weaker relation.}
\label{table}
\end{sidewaystable}

Let us discuss the heuristics behind the dependence of cycle lengths on parameters. It is tempting to think that, if $\theta_j$ is increasing, longer cycles are favored. But this turns out to be incorrect. For instance, typical cycle lengths are of order $n$ for $\theta_j \equiv 1$; they are of order $(\log n)^{1/\gamma}$ for $\theta_j = \e{j^\gamma}$ with $0<\gamma<1$; they are again of order $n$ for $\theta_j = \e{j}$.
The reason for this apparently erratic behavior is best understood from the perspective of statistical mechanics, as already mentioned in \cite{BUV}. We can assign to each index a weight that depends on the length of the cycle it belongs to. Let $L_i(\sigma)$ denote the length of the cycle that contains $i$, and notice the identity
\be
\sum_{i=1}^n a_{L_i(\sigma)} = \sum_{j=1}^n j a_j R_j(\sigma),
\ee
that holds for any $\sigma$ and any numbers $a_1,a_2,\dots$. Choosing $a_j = \frac1j \log\theta_j$, the probability distribution $\bbP_n$ can be rewritten in the form of a ``Gibbs state", namely
\be
\label{stat mech}
\bbP_n(\sigma) = \frac1{n! h_n} \exp \Bigl( \sum_{i=1}^n a_{L_i(\sigma)} \Bigr).
\ee
The weight $a_{L_i(\sigma)}$ plays the r\^ole of the negative of the energy.
The heuristics become
\begin{itemize}
\item If $a_j = \frac1j \log\theta_j$ is increasing, indices prefer to be in longer cycles, and typical cycle lengths are longer.
\item If $a_j = \frac1j \log\theta_j$ is decreasing, the converse happens.
\end{itemize}
Expression \eqref{stat mech} also points to an important symmetry of the parameters: Adding a constant $c$ to each $a_j$ does not affect $\bbP_n$. But it amounts to changing the parameters from $\theta_j$ to $\theta_j \e{cj}$. This also shows that purely exponential parameters are equivalent to uniform random permutations.

We have summarized the results about weighted random permutations in Table \ref{table}. Precise claims can be found in Sections \ref{sec regime A}--\ref{sec regime F}. The behavior of $L_1$ is strikingly similar when the parameters grow sub-exponentially or decay super-exponentially. Notice that typical cycle lengths were obtained earlier in the asymptotic Ewens case \cite{Lugo,BUV} and when $\theta_j$ decreases to 0 \cite{BUV}. We complement these results with statements about finite cycles and total number of cycles (see also \cite{BG}). We also show that the joint distribution of $L_1,L_2,\dots$ converges to the Poisson-Dirichlet distribution.

\section{Generalities}

\subsection{Generating functions and basic expressions for random variables}

Let us start with the exponential generating function of weighted cyclic permutations. There are $(n-1)!$ cyclic permutations of $n$ elements, and therefore
\be
\sum_{n\geq1} \sumtwo{\sigma \in \caS_n}{\text{cyclic}} \frac1{n!} \theta_n z^n = \sum_{n\geq1} \frac{\theta_n}n z^n.
\ee
A general permutation can be viewed as a combinatorial set of cycles, so that the generating function of permutations is given by the exponential of the generating function of cyclic permutations (see Corollary 6.6 of \cite{Aigner}). We set $h_0 = 1$. Then
\be
\label{gen fct}
G_h(z) = \sum_{n\geq0} h_n z^n = \sum_{n\geq0} \frac{z^n}{n!} \sum_{\sigma\in\caS_n} \prod_{j\geq1} \theta_j^{R_j} = \exp \sum_{n\geq1} \frac{\theta_n}n z^n.
\ee
We now obtain expressions that characterize the random variables $L_1,K,R_j$. For part (c), we use $r_{[k]}$ to denote the descending factorial,
\be
r_{[k]} = r (r-1) \dots (r-k+1).
\ee

\begin{proposition}\hfill
\label{prop rvs}
\begin{itemize}
\item[(a)] $\displaystyle \bbP_n(L_1 = j) = \frac{\theta_j h_{n-j}}{n h_n}$.
\item[(b)] $\displaystyle \bbE_n(K) = \sum_{j=1}^n \frac{\theta_j h_{n-j}}{j h_n}$.
\item[(c)] $\displaystyle \bbE_n \Bigl( \prod_{j\geq1} (R_j)_{[k_j]} \Bigr) = \frac{h_{n - \sum_j j k_j}}{h_n} \prod_{j\geq1} \Bigl( \frac{\theta_j}j \Bigr)^{k_j}$ for all integers $k_1, \dots, k_n$ such that $\sum j k_j \leq n$.
\end{itemize}
\end{proposition}

Before proving this proposition, let us mention two useful consequences.
Summing over all possible values for $j$ in the identity (a), we get a relation for the $h_n$s, namely
\be
\label{useful h_n}
h_n = \frac1n \sum_{j=1}^n \theta_j h_{n-j}.
\ee
The following corollary will apply to all regimes of parameters that we consider, except super-exponential growth or decay.

\begin{corollary}
\label{cor Poisson}
When $h_{n-1}/h_n \to 1$ as $n\to\infty$, the joint distribution of the number of finite cycles, $R_1, R_2, R_3, \dots$, converges weakly to independent Poisson with means $\theta_1, \frac{\theta_2}2, \frac{\theta_3}3, \dots$.
\end{corollary}

\begin{proof}
We have from Proposition \ref{prop rvs} (c) that
\be
\lim_{n\to\infty} \bbE_n \Bigl( \prod_{j\geq1} (R_j)_{[k_j]} \Bigr) = \prod_{j\geq1} \Bigl( \frac{\theta_j}j \Bigr)^{k_j}
\ee
for all $k_1,k_2,\dots$ with finitely many nonzero terms. The result is then standard, see e.g.\  Lemma 2.8 of \cite{Worm}.
\end{proof}

\begin{proof}[Proof of Proposition \ref{prop rvs}]
The sum over permutations with $L_1=j$ can be done by first summing over the $(j-1)$ other indices that belong to the cycle that contains 1 (there are $(n-1) \dots (n-j+1)$ possibilities), then by summing over permutations of the remaining $(n-j)$ indices. We get
\be
\bbP_n(L_1=j) = \frac1{n! h_n} (n-1) \dots (n-j+1) \theta_j (n-j)! h_{n-j} = \frac{\theta_j h_{n-j}}{n h_n}.
\ee

We now prove the identity (c). We use the generating function $G_h(s)$.
Let $k_1,k_2,\dots$ be nonnegative integers such that $\sum_j j k_j \leq n$.  Recall that $h_n$ and $G_h$ depend on $\theta_1,\theta_2,\dots$ Using $h_n = [s^n] G_h(s)$, we have
\be
\begin{split}
\bbE_n \Bigl( \prod_{j\geq1} (R_j)_{[k_j]} \Bigr) &= \frac1{h_n} \biggl( \prod_{j\geq1} \theta_j^{k_j} \frac{\dd^{k_j}}{\dd \theta_j^{k_j}} \biggr) h_n \\
&= \frac1{h_n} [s^n] \biggl( \prod_{j\geq1} \theta_j^{k_j} \frac{\dd^{k_j}}{\dd \theta_j^{k_j}} \biggr) G_h(s) \\
&= \frac1{h_n} \prod_{j\geq1} \Bigl( \frac{\theta_j}j \Bigr)^{k_j} [s^{n - \sum_j j k_j}] G_h(s) \\
&= \frac{h_{n - \sum_j j k_j}}{h_n} \prod_{j\geq1} \Bigl( \frac{\theta_j}j \Bigr)^{k_j}.
\end{split}
\ee

As a special case of (c) we have $\bbE_n(R_j) = \frac{\theta_j h_{n-j}}{j h_n}$, which yields the expression (b) for the expected number of cycles.
\end{proof}

\subsection{Saddle point analysis}

Several asymptotic results will be derived using the method of steepest descent, which prompts us to introduce it here. We refer the reader unfamiliar with this method to \cite{Miller} for appropriate background. For our particular application we will require a uniform extension of this method to a family of descent problems essentially indexed by $n$. This extension may more generally be referred to as {\it saddle point analysis}. What this entails and how it is justified will be discussed in the final section. We state here the main result based on this method and how it will apply to the cases of interest in this paper. More details can be found in Section \ref{sec uniform SPE}.

Let $G(z)$ be a function that is analytic at the origin with Taylor series there having a finite, non-zero radius of convergence (which we will take to be $1$ in all cases).  For $r>0$, let
\be
\begin{split}
&\alpha(r) = r (\log G(r))', \\
&\beta(r) = r \alpha'(r)
\end{split}
\ee
and assume that
$$\lim_{r \to 1} \alpha(r) = \infty; \,\,\, \lim_{r \to 1} \beta(r) = \infty.$$
Assume further that for $n$ sufficiently large there is an $\epsilon$ such that on  $(1 - \epsilon,1)$ there is a unique solution to the equation 
\be
\alpha(r)  = n
\ee
and denote this root by  $r_n$. Notice that  in this range $r_n$ is increasing and $\lim_{n\to\infty} r_n = 1$. Then the main saddle point result we use states that 
\be
\label{general saddle asymptotic}
[z^n] G(z) = \frac{G(r_n)}{r_n^n \sqrt{2\pi \beta(r_n)}} \bigl( 1 + o(1) \bigr).
\ee

We will apply this in two cases: (i) for the generating function $G_\theta$ of the parameters; (ii) for the generating function $G_h$ of the coefficients $h_n$. 

It is convenient to introduce
\be
I_\mu(z) = \sum_{n\geq1} n^\mu \theta_n z^n.
\ee
These functions satisfy the following recursion relations
\be
\label{recursion relations}
I_\mu(z) = z \, I_{\mu-1}'(z).
\ee

In the case of the parameter generating functions we take
\be
G_\theta(z) = z^{-1}I_0(z) = \frac{d}{dz} I_{-1}(z).
\ee
One easily finds that 
\be
\alpha(z) = \frac{I_1(z)}{I_0(z)} -1.
\ee
Define $\rho_n$ by the equation $\alpha(\rho_n) = n$. Then, applying (\ref{general saddle asymptotic}), the parameters are asymptotically equal to
\be
\label{asymptotic theta}
\theta_{n+1} = \frac{G_\theta(\rho_n)}{\rho_n^{n + \frac12} \sqrt{2\pi \alpha'(\rho_n)}} \bigl( 1 + o(1) \bigr).
\ee

Let us turn to the second case, i.e., the generating function of $h_n$ given by \eqref{gen fct}.
We have the relations
\be
\label{Gh}
\begin{split}
&G_h(z) = \exp I_{-1}(z), \\
&\alpha(r) = I_0(r), \quad I_0(r_n) = n, \\
&\beta(r) = I_1(r).
\end{split}
\ee
Notice that $r_n$ is increasing in $n$ and that $\lim r_n$ is equal to the radius of convergence of $I_0$.
The coefficients $h_n$ are asymptotically equal to
\be
\label{asymptotic h}
h_n = \frac{\e{I_{-1}(r_n)}}{r_n^n \sqrt{2\pi I_1(r_n)}} \bigl( 1 + o(1) \bigr).
\ee
We will actually deal with ratios of those numbers, and the following bounds will prove extremely useful.

\begin{proposition}
\label{prop great bounds}
Assume that $\theta_1,\theta_2,\dots$ are such that the saddle point approximation involving $G_h$ is valid. Then
\[
\sqrt{\frac{I_1(r_n)}{I_1(r_{n-j})}} r_{n-j}^j \leq \frac{h_{n-j}}{h_n} \bigl( 1 + o(1) \bigr) \leq \sqrt{\frac{I_1(r_n)}{I_1(r_{n-j})}} r_n^j.
\]
\end{proposition}

\begin{proof}
The asymptotic approximation \eqref{asymptotic h} for $h_n$ implies that
\be
\frac{h_{n-j}}{h_n} = \sqrt{\frac{I_1(r_n)}{I_1(r_{n-j})}} \e{\Delta(n,j)} \bigl( 1 + o(1) \bigr),
\ee
where $\Delta(n,j)$ can be written as
\be
\Delta(n,j) = I_{-1}(r_{n-j}) - I_0(r_{n-j}) \log r_{n-j} - I_{-1}(r_n) + I_0(r_n) \log r_n.
\ee
By the fundamental theorem of calculus, using \eqref{recursion relations},
\be
\Delta(n,j) = \int_{r_n}^{r_{n-j}} \frac{\dd}{\dd u} \Bigl( I_{-1}(u) - I_0(u) \log u \Bigr) \dd u = \int_{r_{n-j}}^{r_n} \frac{I_1(u) \log u}u \dd u.
\ee
We can bound $\log u \geq \log r_{n-j}$. The integral of $I_1(u)/u$ yields $I_0(r_n) - I_0(r_{n-j}) = j$ and we get the lower bound of the proposition. The upper bound is similar, using $\log u \leq \log r_n$.
\end{proof}

\section{Parameters with super-exponential growth}
\label{sec regime A}

First we consider the regime  when $\theta_n$ diverges fast enough. It is not too hard to show that only one cycle of length $n$ is present, meaning also that all finite cycles have disappeared.

\begin{theorem}
\label{thm large parameters}
Assume that $\theta_n>0$ for all $n$, and that
\[
\lim_{n\to\infty} \sum_{j=1}^{n-1} \frac{\theta_j \theta_{n-j}}{\theta_n} = 0.
\]
Then
\[
\lim_{n\to\infty} \bbP_n(L_1=n) = 1.
\]
\end{theorem}

It immediately follows that $\bbP_n(K=1) \to 1$ and $\bbP_n(R_j=0) \to 1$ for all fixed $j$.

Let us check that the theorem applies to the parameters $\theta_n = \e{n^\gamma}$ with $\gamma>1$. We have
\be
\frac{\theta_j \theta_{n-j}}{\theta_n} = \e{-n^\gamma [1 - (\frac jn)^\gamma - (1-\frac jn)^\gamma]}.
\ee
It is easy to check that $(1-s)^\gamma \leq 1-cs$ for $0 \leq s \leq \frac12$ with $c = 2 (1 - 2^{-\gamma}) > 1$. Then for $j \leq \frac n2$,
\be
\frac{\theta_j \theta_{n-j}}{\theta_n} \leq \e{-n^\gamma [ c \frac jn - (\frac jn)^\gamma]} \leq \e{-(c-1) n^{\gamma-1} j}.
\ee
It follows that
\be
\sum_{j=1}^{n-1} \frac{\theta_j \theta_{n-j}}{\theta_n} \leq 2 \sum_{j=1}^{n/2} \frac{\theta_j \theta_{n-j}}{\theta_n} \leq 2 \sum_{j\geq1} \e{-(c-1) n^{\gamma-1} j},
\ee
which clearly goes to 0 as $n\to\infty$.

\begin{proof}[Proof of theorem \ref{thm large parameters}]
By the assumption of the theorem, there exists $N$ such that $\theta_j \theta_{n-j} < \theta_n$ for all $n \geq N$. Let $C$ such that $h_n \leq C \theta_n$ for all $n \leq N$. We now prove by induction that this upper bound holds for all $n$. By \eqref{useful h_n} and the induction hypothesis,
\be
h_{n+1} \leq \frac C{n+1} \sum_{j=1}^{n+1} \theta_j \theta_{n+1-j} \leq C \theta_{n+1}.
\ee

We have $\bbP_n(L_1=n) = \frac{\theta_n}{n h_n}$ by Proposition \ref{prop rvs} (a). Using \eqref{useful h_n}, we get
\be
\frac{n h_n}{\theta_n} = \sum_{j=1}^n \frac{\theta_j h_{n-j}}{\theta_n} = 1 + \sum_{j=1}^{n-1} \frac{\theta_j h_{n-j}}{\theta_n}.
\ee
Using $h_{n-j} \leq C \theta_{n-j}$ and the assumption of the theorem, we see that $\frac{n h_n}{\theta_n} \to 1$ as $n\to\infty$.
\end{proof}

\section{Parameters with sub-exponential growth}
\label{sec regime B}

Our goal here is to understand the regime of parameters that grow sub-exponentially, $\theta \approx \e{n^\gamma}$ with $0<\gamma<1$. It turns out to be difficult to tackle this case directly and we appeal to an indirect approach, by focusing on the generating function rather than its parameters.

Let $A,a,b,c$ be positive parameters to be chosen later, and let
\be
\label{Nick's subtle generating function}
G_\theta(z) = A (1-z)^{-c} \e{a (1-z)^{-b}}.
\ee
Then the parameters, $\theta_n$, are the coefficients of $I_0$ or, equivalently, the shifted coefficients of $G_\theta$:
\be
\theta_n = [z^{n-1}] G_\theta(z).
\ee
It is not hard to check (by repeated differentiation) that $\theta_n>0$ for all $n\geq1$. Notice also that the radius of convergence of $G_\theta$ is 1.

\begin{proposition}
\label{prop sub-exp parameters}
\[
\theta_{n+1} = \frac{A (ab)^{\frac1{b+1} (\frac12-c)}}{\sqrt{2\pi (b+1)}} n^{\frac1{b+1} (c - \frac{b+2}2)} \exp \Bigl\{ \bigl[ a(b+1) (ab)^{-\frac b{b+1}}\bigr] n^{\frac b{b+1}} + \bigl[  \tfrac12 (ab)^{\frac2{b+1}}\bigr] n^{\frac{b-1}{b+1}}  + o \bigl( n^{\frac{b-1}{b+1} \vee 0} \bigr) \Bigr\}.
\]
\end{proposition}

\begin{proof}
We use the saddle point method, which has long been used for this class of functions \cite{Wright}. We have
\be
\begin{split}
&\alpha(z) = z\frac{d}{dz} \log G_\theta(z) = z  \left[ \frac{c}{1-z} + \frac{ab}{\left( 1-z \right)^{b+1}} \right],\\
&I_1(z) =  I_0(z)\left(\alpha(z) + 1\right) =  I_0(z) \bigl[ ab z (1-z)^{-b-1} + cz (1-z)^{-1} + 1 \bigr].
\end{split}
\ee
Notice that, as $z\to1$,
\be
\label{asymptotic for a'}
\alpha'(z) = ab (b+1) (1-z)^{-b-2} (1+o(1)).
\ee
Defining $\rho_n$ by $\alpha(\rho_n) = n$ we obtain
\be
\rho_n = 1 - \bigl( \tfrac{ab}n \bigr)^{\frac1{b+1}} \bigl[ \rho_n + \tfrac c{ab} (1-\rho_n)^b - \tfrac{c}{ab} (1-\rho_n)^{b+1} \bigr]^{\frac{1}{b+1}} .
\ee
By the implicit function theorem, we get
\be
\begin{split}
\rho_n &= 1 - \bigl( \tfrac{ab}n \bigr)^{\frac1{b+1}}  + \tfrac{1}{b+1} \bigl( \tfrac{ab}n \bigr)^{\frac2{b+1}} - \tfrac c{b+1} \bigl( \tfrac{1}n \bigr) + o\bigl( n^{-(1 \wedge \frac 2{b+1})} \bigr)  \\
&= \exp\Bigl\{ - \bigl( \tfrac{ab}n \bigr)^{\frac1{b+1}} + \tfrac12 \tfrac{1-b}{1+b}\bigl( \tfrac{ab}n \bigr)^{\frac2{b+1}} - \tfrac c{b+1} \bigl( \tfrac{1}n \bigr) + o\bigl( n^{-(1 \wedge \frac 2{b+1})} \bigr)\Bigr\}.
\end{split}
\ee
We compute the terms that appear in the formula \eqref{asymptotic theta} for $\theta_n$. First,
\be
\begin{split}
&G_\theta(\rho_n) = A \bigl( \tfrac n{ab} \bigr)^{\frac c{b+1}} \exp\Bigl\{ a \bigl( \tfrac n{ab} \bigr)^{\frac b{b+1}} + \tfrac{ab}{b+1} \bigl( \tfrac n{ab} \bigr)^{\frac{b-1}{b+1}} - \tfrac{c}{b+1} + o\bigl( n^{\frac{b-1}{b+1} \vee 0} \bigr) \Bigr\}, \\
&\rho_n^{-n} = \exp\Bigl\{ (ab)^{\frac1{b+1}} n^{\frac b{b+1}} + \tfrac12 \tfrac{b-1}{b+1} (ab)^{\frac2{b+1}} n^{\frac{b-1}{b+1}} + \tfrac{c}{b+1} 
+ o\bigl( n^{\frac{b-1}{b+1} \vee 0} \bigr) \Bigr\}, \\
&\alpha'(\rho_n) = (b+1) (ab)^{-\frac1{b+1}} n^{\frac{b+2}{b+1}} (1+o(1)).
\end{split}
\ee
We get the proposition by inserting these values into \eqref{asymptotic theta}.
\end{proof}

We choose the numbers $A,a,b,c$ so that $\theta_n \approx \e{n^\gamma}$. Precisely, let
\be
\label{the right choice}
\begin{split}
&b = \tfrac\gamma{1-\gamma}, \\
&a = (1-\gamma) \gamma^{\frac\gamma{1-\gamma}}, \\
&c= \tfrac b2 + 1, \\
&A = \sqrt{2\pi (b+1)} (ab)^{-\frac1{b+1} (\frac12 - c)}.
\end{split}
\ee
They imply the following relations, that are often useful when checking the details of the calculations:
\be
\gamma = \tfrac b{b+1}, \qquad ab = \gamma^{\frac1{1-\gamma}}, \qquad \tfrac1{b+1} = 1-\gamma.
\ee
With these numbers, the precise asymptotic expression of $\theta_n$ is
\be
\label{precise theta}
\theta_n = \exp\Bigl\{ n^\gamma + \tfrac12 \gamma^2 n^{2\gamma-1} + o \bigl( n^{(2\gamma-1) \vee 0} \bigr) \Bigr\}.
\ee
Notice that $\theta_n = \e{n^\gamma} (1+o(1))$ when $\gamma<\frac12$. The case $\gamma=\frac12$ can be handled by modifying the number $A$. For the case $\gamma>\frac12$, the correction $n^{2\gamma-1}$ is present in the exponential and it cannot be removed easily.

It is time to state the main result of this section.

\begin{theorem}
Consider the set of parameters $\theta_1, \theta_2, \dots$ whose generating function is given by Eq.\ \eqref{Nick's subtle generating function} with $A, a, b, c$ specialized as in  Eq.\ \eqref{the right choice}. Then
\begin{itemize}
\item[(a)] $\displaystyle \frac{L_1}{(\log n)^{1/\gamma}} \Rightarrow (1-\gamma)^{-1/\gamma}$.
\item[(b)] $R_1,R_2, \dots$ converge weakly to independent Poisson random variables with respective means $\theta_1, \frac{\theta_2}2, \dots$.
\end{itemize}
\end{theorem}

\begin{proof}
Let $B = (1-\gamma)^{-1/\gamma}$.
We show that for any $\varepsilon>0$,
\be
\lim_{n\to\infty} \bbP_n \Bigl( \Bigl| \frac{L_1}{(\log n)^{1/\gamma}} - B \Bigr| > \varepsilon \Bigr) = 0.
\ee
By Proposition \ref{prop rvs} (a), we have
\be
\bbP_n \Bigl( \Bigl| \frac{L_1}{(\log n)^{1/\gamma}} - B \Bigr| > \varepsilon \Bigr) = \sum_{j : |\frac j{(\log n)^{1/\gamma}} - B| > \varepsilon} \frac{\theta_j h_{n-j}}{n h_n}.
\ee
We use the saddle point method for the generating function $G_h = \e{I_{-1}}$. The equation $I_0(r_n)=n$ implies that
\be
\label{r_n}
r_n = 1 - \bigl[ \tfrac1a \log \tfrac{n (1-r_n)^c}{A r_n} \bigr]^{-1/b}.
\ee
(We keep using $A,a,b,c$ rather than $\gamma$ for convenience.) By the implicit function theorem, we get
\be
\label{subexpscale}
\begin{split}
r_n &= 1 - a^{\frac1b} (\log n)^{-\frac1b} + O\bigl( \tfrac{\log\log n}{(\log n)^{b+1}} \bigr) \\
&= \exp\Bigl\{ - a^{\frac1b} (\log n)^{-\frac1b} + O\bigl( \tfrac{\log\log n}{(\log n)^{b+1}} \vee (\log n)^{-\frac2b} \bigr) \Bigr\}.
\end{split}
\ee
It follows that $I_1(r_n) = n (\frac1a \log n)^{1/\gamma} (1+o(1))$. We use Proposition \ref{prop great bounds} to get
\be
\frac{h_{n-j}}{h_n} \leq \sqrt{\tfrac n{n-j}} \bigl( \tfrac{\log n}{\log(n-j)} \bigr)^{1/2\gamma} \exp\Bigl\{ -j a^{\frac1b} (\log n)^{-\frac1b} \bigl( 1 + O\bigl( \tfrac{\log\log n}{\log n} \vee \tfrac1{(\log n)^{1/b}} \bigr) \bigr) \Bigr\}.
\ee
The cases $j=n$ and $j=n-1$ need actually to be handled separately. Using the expression \eqref{precise theta} for $\theta_j$ and the bound above, it is easy to check that
\be
\lim_{n\to\infty} \sum_{j=n/2}^n \frac{\theta_j h_{n-j}}{n h_n} = 0.
\ee
For $1 \leq j \leq n/2$, we have
\bm
\sum_{j : |\frac j{(\log n)^{1/\gamma}} - B| > \varepsilon} \frac{\theta_j h_{n-j}}{n h_n} \leq C \sum_{j : |\frac j{(\log n)^{1/\gamma}} - B| > \varepsilon} \exp\Bigl\{ j^\gamma - j a^{\frac1b} (\log n)^{-\frac1b} - \log n \\
+ O(j^{(2\gamma-1) \vee 0}) + O\bigl( \tfrac{\log\log n}{\log n} \vee \tfrac1{(\log n)^{1/b}} \bigr) \Bigr\}.
\end{multline}
Let us make the change of variables $j = i (\log n)^{1/\gamma}$. Then
\be
\sum_{j : |\frac j{(\log n)^{1/\gamma}} - B| > \varepsilon} \frac{\theta_j h_{n-j}}{n h_n} \leq C \sumtwo{i \in (\log n)^{-1/\gamma} \bbN}{|i - B| > \varepsilon} \e{-\log n [a^{\frac1b} i - i^\gamma + 1+ o(1)]}.
\ee
It is easy to see that the function $f(x) = a^{\frac1b} x - x^\gamma + 1$ is convex with a minimum at $B = (1-\gamma)^{-1/\gamma}$, where it takes value 0. For $|x-B|>\varepsilon$ we can bound $f(x) > \delta |x-B|$. We can estimate the sum by an integral, in order to get
\be
\sum_{j : |\frac j{(\log n)^{1/\gamma}} - B| > \varepsilon} \frac{\theta_j h_{n-j}}{n h_n} \leq C (\log n)^{1/\gamma} \int_{|x-B|>\varepsilon} \e{-(\log n) \delta |x-B|} \dd x,
\ee
which clearly vanishes in the limit $n\to\infty$. This proves (a).

It is clear from Proposition \ref{prop great bounds} and Eq.\ \eqref{subexpscale} that $h_{n-1}/h_n \to 1$ as $n\to \infty$, so that (b) follows immediately from Corollary \ref{cor Poisson}.
\end{proof}

\section{Parameters with algebraic growth}
\label{sec regime C}

We again work with a generating function rather than parameters. Recall that $\gamma>0$, and let
\be
\label{algebraic gen fct}
I_0(z) = \frac{\Gamma(\gamma+1)}{(1-z)^{\gamma+1}} - \Gamma(\gamma + 1).
\ee
One easily checks that
\be
\frac{\dd^n}{\dd z^n} I_0(z) = \frac{\Gamma(n+\gamma+1)}{(1-z)^{\gamma+n+1}}.
\ee
One then gets the parameters:
\be
\label{algparams}
\theta_n = [z^n] I_0(z) = \frac{1}{n!} \frac{\dd^n}{\dd z^n} I_0(0) = \frac{\Gamma(\gamma+n+1)}{n!}.
\ee
By a straightforward application of Stirling's formula one sees that the  parameters grow algebraically:
\be
\theta_n = n^\gamma (1+o(1)).
\ee

\begin{theorem}
Choose $\theta_1,\theta_2,\dots$ such that their generating function is given by Eq.\ \eqref{algebraic gen fct}. Then
\begin{itemize}
\item[(a)] $L_1 / n^{\frac1{1+\gamma}}$ converges weakly to the Gamma random variable with parameters $(\gamma+1,a)$ with $a = \Gamma(\gamma+1)^{\frac1{\gamma+1}}$. In other words, we have
\[
\lim_{n\to\infty} \bbP_n \Bigl( \frac{L_1}{n^{1/(1+\gamma)}} < s \Bigr) = \int_0^s x^\gamma \e{-a x} \dd x.
\]
\item[(b)] $\displaystyle \lim_{n\to\infty} n^{-\frac\gamma{\gamma+1}} \bbE_n(K) = \bigl( \Gamma(\gamma) / \gamma^\gamma \bigr)^{\frac1{\gamma+1}}$.
\item[(c)] The distribution of number of finite cycles converges weakly to independent Poisson random variables with means $\theta_1, \frac{\theta_2}2, \dots$.
\end{itemize}
\end{theorem}

\begin{proof}
We use the saddle point method. Let $r_n$ be defined by $I_0(r_n) = n$. Then
\be
r_n = 1 - \bigl( \tfrac n{\Gamma(1+\gamma)} + 1 \bigr)^{-\frac1{1+\gamma}}.
\ee
It is enough for our purpose to retain 
\be
\label{algscale}
r_n = 1 - a n^{-\frac1{1+\gamma}} + O(n^{-\frac{2+\gamma}{1+\gamma}}) = \exp\bigl\{ -a n^{-\frac1{1+\gamma}} + O(n^{-\frac2{1+\gamma}}) \bigr\}.
\ee
In order to use Proposition \ref{prop great bounds}, we check that $r_{n-j}^j$ is close to $r_n^j$. We assume from now on that $j < C n^{1/(1+\gamma)}$ for a constant $C$ independent of $n$. A few calculations yield
\be
r_{n-j}^j = \exp\bigl\{ -a j n^{-\frac1{1+\gamma}} + O(j^2 n^{-\frac{2+\gamma}{1+\gamma}} \vee j n^{-\frac2{1+\gamma}}) \bigr\}.
\ee
Then
\be
r_{n-j}^j = r_n^j (1+o(1)).
\ee
We also have
\be
\label{algI1}
I_1(z) = \Gamma(\gamma + 2) \, z \, (1-z)^{-\gamma-2}.
\ee
Then
\be
\begin{split}
\frac{I_1(r_n)}{I_1(r_{n-j})} &= \frac{r_n}{r_{n-j}} \Bigl( \frac{n }{n-j} \Bigr)^{\frac{2+\gamma}{1+\gamma}} \left(1 + o(1)\right)\\
&= \e{\frac{aj}{1+\gamma}n^{-\frac{\gamma + 2}{\gamma + 1}}}\left(1 + o(1)\right)
\end{split}
\ee
It follows that
\be
\begin{split}
\frac{h_{n-j}}{h_n} &= r_n^j \sqrt{\frac{I_1(r_n)}{I_1(r_{n-j})}}\bigl( 1 + o(1) \bigr)\\
\label{algebraic ratio} &=  \e{-a j n^{-1/(1+\gamma)}} \bigl( 1 + o(1) \bigr).
\end{split}
\ee
We can now proceed to the calculation of the distribution of $L_1$. Using Proposition \ref{prop rvs} (a), we have
\be
\bbP_n \Bigl( \frac{L_1}{n^{1/(1+\gamma)}} < s \Bigr) = \sum_{j=1}^{s \, n^{1/(1+\gamma)}} \frac{j^\gamma}n \e{-a j n^{-1/(1+\gamma)}} \bigl( 1 + o(1) \bigr).
\ee
(We can use the asymptotic value for $\theta_j$ because finite $j$ contribute a vanishing amount.) We rescale the variables in order to recognize a Riemann integral:
\be
\bbP_n \Bigl( \frac{L_1}{n^{1/(1+\gamma)}} < s \Bigr) = \frac1{n^{\frac1{1+\gamma}}} \sum_{j=1}^{s \, n^{1/(1+\gamma)}} \Bigl( \frac j{n^{1/(1+\gamma)}} \Bigr)^\gamma \e{-a \frac j{n^{1/(1+\gamma)}}} \bigl( 1 + o(1) \bigr).
\ee
As $n\to\infty$, this converges to the probability that the Gamma random variable with parameters $(\gamma+1,a)$ be less than $s$.

For part (b) we use Proposition \ref{prop rvs} (b). Using $\theta_j = j^\gamma (1+o(1))$ and Eq.\ \eqref{algebraic ratio}, we have
\be
\bbE_n(K) = \sum_{j=1}^n j^{\gamma-1} \e{-a j n^{1/(\gamma+1)}} \bigl( 1 + o(1) \bigr).
\ee
Notice that the contribution of finite $j$ vanishes, which justifies using the asymptotic expression for $\theta_j$. Introducing the appropriate scaling that leads to a Riemann integral, we rewrite the expression as
\be
\frac{a^\gamma}{n^{\frac\gamma{\gamma+1}}} \bbE_n(K) = \frac a{n^{\frac1{\gamma+1}}} \sum_{j=1}^n \Bigl( \frac{aj}{n^{\frac1{\gamma+1}}} \Bigr)^{\gamma-1} \e{-a j n^{-\frac1{\gamma+1}}} \bigl( 1 + o(1) \bigr).
\ee
The right side converges to $\int_0^\infty x^{\gamma-1} \e{-x} \dd x = \Gamma(\gamma)$ and we obtain the claim (b).

Part (c) follows from \eqref{algebraic ratio} and Corollary \ref{cor Poisson}.
\end{proof}

\section{Asymptotic Ewens parameters}
\label{sec regime D}

Past studies of the Ewens distribution have focused on the number of cycles. It was shown in particular that the number of cycles with length less than $n^s$ is approximately equal to $\theta s \log n$ for all $0<s\leq1$, and that it satisfies a central limit theorem \cite{Han} and a large deviation principle \cite{FH}. In this section we consider the case where $\theta_j \to \theta$ as $j\to\infty$. We look at the distribution of finite cycles and at the joint distribution of the largest cycles. Let $L^{(1)}, L^{(2)}, \dots$ denote the cycle lengths in nonincreasing order (for all $\sigma \in \caS_n$ we have $\sum_j j R_j(\sigma) = \sum_i L^{(i)} = n$).

The large cycle lengths converge to the Poisson-Dirichlet distribution. In order to define it, first consider a sequence of i.i.d.\ beta random variables with parameters $(1,\theta)$, $(X_1,X_2,\dots)$. That is, $\bbP(X>s) = (1-s)^\theta$ for $0\leq s\leq1$. Then form the sequence $(X_1, (1-X_1) X_2, (1-X_1) (1-X_2) X_3, \dots)$. It is not hard to check that it is a random partition of $[0,1]$, which is called the Griffiths-Engen-McCloskey distribution. Reorganizing these numbers in nonincreasing order gives another random partition of $[0,1]$, and the corresponding distribution is called Poisson-Dirichlet.

\begin{theorem}
\label{thm Ewens}
Assume that $\theta_n \to \theta$. Then, as $n\to\infty$,
\begin{itemize}
\item[(a)] the random variables $R_1, R_2, R_3, \dots$ converge weakly to independent Poisson with respective means $\theta_1, \frac{\theta_2}2, \frac{\theta_3}3, \dots$;
\item[(b)] the total number of cycles is logarithmic: $\displaystyle \lim_{n\to\infty} \frac{\bbE_n(K)}{\log n} = \theta$;
\item[(c)] the joint distribution of $\frac{L^{(1)}}n, \frac{L^{(2)}}n, \dots$ converges weakly to Poisson-Dirichlet with parameter $\theta$.
\end{itemize}
\end{theorem}

The last result involves only the limit $\theta$ and not the individual parameters $\theta_j$s. This is not surprising as the longest cycles become infinite as $n\to\infty$. The theorems of \cite{Han,FH} also concern cycles of diverging lengths and they should remain valid in the asymptotic Ewens case without modifications. On the other hand, the distribution of finite cycles depends explicitly on the $\theta_j$s.

The rest of this section is devoted to the proof of this theorem. It relies on estimates for the normalization $h_n$. Let us introduce the function $\Lambda(x)$, $x\geq1$, by
\be
\label{def Lambda}
\Lambda \Bigl( \frac1{1-s} \Bigr) = \exp \sum_{j\geq1} \frac{\theta_j - \theta}j s^j,
\ee
where $0 \leq s < 1$.

\begin{lemma}
\label{lem Lambda}
The function $\Lambda$ is ``slowly varying" in a strong sense. Namely, let $(x_n)$ and $(y_n)$ be any two diverging sequences such that there exists a constant $C>1$ with
\[
\frac1C \leq \frac{x_n}{y_n} \leq C
\]
for all $n$. Then
\[
\lim_{n\to\infty} \frac{\Lambda(x_n)}{\Lambda(y_n)} = 1.
\]
\end{lemma}

\begin{proof}
We need to show that
\be
\label{must be 0}
\sum_{j\geq1} \frac{\theta_j - \theta}j \Bigl[ \bigl( 1 - \tfrac1{x_n} \bigr)^j - \bigl( 1 - \tfrac1{y_n} \bigr)^j \Bigr]
\ee
converges to 0 as $n\to\infty$. Given $\varepsilon>0$, let $N_\varepsilon$ such that $|\theta_j - \theta| < \varepsilon$ for all $j>N_\varepsilon$. The sum over the first $N_\varepsilon$ terms of \eqref{must be 0} goes to 0 as $n\to\infty$. The rest is less than
\be
\varepsilon \sum_{j\geq1} \frac1j \Bigl| \bigl( 1 - \tfrac1{x_n} \bigr)^j - \bigl( 1 - \tfrac1{y_n} \bigr)^j \Bigr| = \varepsilon \Bigl| \log \frac{x_n}{y_n} \Bigr| \leq \varepsilon \log C.
\ee
The expression \eqref{must be 0} is then as small as we want when $n$ is large enough.
\end{proof}

Using the definition \eqref{def Lambda} and recognizing the Taylor series of the logarithm, we have
\be
\label{gen fct 2}
G_{h}(s) = (1-s)^{-\theta} \, \Lambda \Bigl( \frac1{1-s} \Bigr).
\ee

\begin{proposition}
\label{prop h_n}
\[
h_n = \frac{n^{\theta-1}}{\Gamma(\theta)} \, \Lambda(n) \, \bigl( 1 + o(1) \bigr).
\]
\end{proposition}

\begin{proof}
The generating function of $h_n$ being given by \eqref{gen fct 2} with $\Lambda$ a slowly varying function, we can use the Tauberian theorem of Hardy-Littlewood-Karamata (see Theorem 9 of \cite{Bender}) to obtain
\be
\label{Tauberian result}
\frac1n \sum_{j=0}^{n-1} h_j = \frac{n^{\theta-1} \Lambda(n)}{\Gamma(\theta+1)} \, \bigl( 1 + o(1) \bigr).
\ee
We need to remove the Ces\`aro average in the left side. From \eqref{useful h_n}, we have
\be
h_n = \frac\theta n \sum_{j=0}^{n-1} h_j + \frac1n \sum_{j=0}^{n-1} (\theta_{n-j} - \theta) h_j.
\ee
The first term of the right side can be combined with \eqref{Tauberian result} and it gives the right result. We need to check that the correction due to the second term is irrelevant, i.e., we need to check that
\be
\lim_{n\to\infty} \frac1{n^\theta \Lambda(n)} \sum_{j=0}^{n-1} (\theta_{n-j} - \theta) h_j = 0.
\ee
Let $\varepsilon>0$. Using \eqref{Tauberian result}, we first have
\be
\frac1{n^\theta \Lambda(n)} \sum_{j=0}^{(1-\varepsilon) n} |\theta_{n-j} - \theta| h_j \leq \Bigl( \sup_{j \geq \varepsilon n} |\theta_j - \theta| \Bigr) \frac{((1-\varepsilon) n)^\theta \Lambda((1-\varepsilon) n)}{n^\theta \Lambda(n) \Gamma(\theta+1)} \, \bigl( 1 + o(1) \bigr),
\ee
which clearly vanishes in the limit $n\to\infty$. Second, let $C = \sup_j |\theta_j - \theta|$, and observe that
\be
\sum_{j = (1-\varepsilon) n}^{n-1} h_j = \sum_{j=0}^{n-1} h_j - \sum_{j=0}^{(1-\varepsilon) n} h_j = \frac{n^\theta \Lambda(n)}{\Gamma(\theta+1)} \Bigl[ 1 + o(1) - (1-\varepsilon)^\theta \frac{\Lambda((1-\varepsilon)n)}{\Lambda(n)} \, \bigl( 1 + o(1) \bigr) \Bigr].
\ee
Then
\be
\limsup_{n\to\infty} \frac1{n^\theta \Lambda(n)} \sum_{j = (1-\varepsilon) n}^{n-1} |\theta_{n-j} - \theta| h_j \leq \frac C{\Gamma(\theta+1)} \bigl[ 1 - (1-\varepsilon)^\theta \bigr],
\ee
which is arbitrarily small since $\varepsilon$ is arbitrary.
\end{proof}

We can now prove the theorem.

\begin{proof}[Proof of Theorem \ref{thm Ewens}]
The claim (a) easily follows from Lemma \ref{lem Lambda}, Proposition \ref{prop h_n}, and Corollary \ref{cor Poisson}.

For the claim (b), we first observe that the number of cycles of length larger than $\frac n{\sqrt{\log n}}$ is less than $\sqrt{\log n}$, so we only need to consider smaller cycles. This means that we can sum up to $\frac n{\sqrt{\log n}}$ in the expression of Proposition \ref{prop rvs} (b) for $\bbE_n(K)$.
By Proposition \ref{prop h_n}, the ratio $\frac{h_{n-j}}{h_n}$ converges to 1 as $n\to\infty$, uniformly in $1 \leq j \leq n/\sqrt{\log n}$. It follows that
\be
\lim_{n\to\infty} \frac{\bbE_n(K)}{\log n} = \lim_{n\to\infty} \frac1{\log n} \sum_{j=1}^{n / \sqrt{\log n}} \frac{\theta_j}j = \theta.
\ee

We turn to part (c).
Let $\tilde L_1, \tilde L_2,\dots$ denote the lengths of the cycles when they have been ordered e.g.\ according to their smallest element. That is, $\tilde L_1 = L_1$ is the length of the cycle that contains the index 1; $\tilde L_2$ is the length of the cycle that contains the smallest index that is not in the first cycle; and so on... We show that for all $k$,
\[
\Bigl( \frac{\tilde L_1}n, \frac{\tilde L_2}{n-\tilde L_1}, \frac{\tilde L_3}{n - \tilde L_1 - \tilde L_2}, \dots, \frac{\tilde L_k}{n - \tilde L_1 - \dots - \tilde L_{k-1}} \Bigr)
\]
converges to i.i.d.\ beta random variables with parameters $(1,\theta)$. This implies that $(\frac{\tilde L_1}n, \frac{\tilde L_2}n, \dots)$ converges weakly to GEM$(\theta)$; and reordering the cycle lengths in nonincreasing order yields PD$(\theta)$. It is enough to show that for any $k$ and any $a_1, \dots, a_k \in (0,1)$, we have
\be
\lim_{n\to\infty} \bbP_n \Bigl( \frac{\tilde L_1}n \leq a_1, \dots, \frac{\tilde L_k}{n - \tilde L_1 - \dots - \tilde L_{k-1}} \leq a_k \Bigr) = \prod_{i=1}^k \bigl[ 1 - (1-a_i)^\theta \bigr].
\ee
The right side is the beta measure of the product of intervals $\times_{i=1}^k (0,a_i)$.

We proceed by induction on $k$, starting with $k=1$. By Proposition \ref{prop rvs} (a), we have
\be
\bbP_n \Bigl( \frac{L_1}n \leq a_1 \Bigr) = \frac1n \sum_{j=0}^{a_1 n} \theta_j \frac{h_{n-j}}{h_n}
= \frac1n \sum_{j=0}^{a_1 n} \theta_j \bigl( 1 - \tfrac jn \bigr)^{\theta-1} \frac{\Lambda(n-j)}{\Lambda(n)} \, \bigl( 1 + o(1) \bigr).
\ee
We used Proposition \ref{prop h_n} to get the second identity.
By Lemma \ref{lem Lambda} the ratio $\frac{\Lambda(n-j)}{\Lambda(n)}$ converges to 1. We clearly have a Riemann sum, so that
\be
\lim_{n\to\infty} \bbP_n \Bigl( \frac{L_1}n \leq a_1 \Bigr) = \theta \int_0^{a_1} (1-x)^{\theta-1} \dd x = 1 - (1-a_1)^\theta.
\ee

Next, we assume that the claim has been proved for $k$ and we prove it for $k+1$. Let
\be
A = \Bigl\{ (\ell_1,\dots,\ell_k) \in \{1,\dots,n\}^k : \frac{\ell_1}n \leq a_1, \dots, \frac{\ell_k}{n - \ell_1 - \dots - \ell_{k-1}} \leq a_k \Bigr\}.
\ee
It is not hard to verify that, on $A$,
\be
\label{numbers diverge}
n - \ell_1 - \dots - \ell_k \geq n \prod_{i=1}^k (1-a_i).
\ee
We have
\bm
\label{conditional prob}
\bbP_n \Bigl( (\tilde L_1, \dots, \tilde L_k) \in A, \tfrac{\tilde L_{k+1}}{n - \tilde L_1 - \dots - \tilde L_k} \leq a_{k+1} \Bigr) \\
= \sum_{(\ell_1, \dots, \ell_k) \in A} \bbP_n( \tilde L_1 = \ell_1, \dots, \tilde L_k = \ell_k ) \,\, \bbP_n \Bigl( \tfrac{\tilde L_{k+1}}{n - \tilde L_1 - \dots - \tilde L_k} \leq a_{k+1} \Big| \tilde L_1 = \ell_1, \dots, \tilde L_k = \ell_k \Bigr).
\end{multline}
Now we use the self-similarity of weighted permutations: Having chosen the first $k$ cycles, the distribution of the $(k+1)$th cycle is identical but with less indices available. Precisely, we have
\be
\bbP_n \Bigl( \tfrac{\tilde L_{k+1}}{n - \tilde L_1 - \dots - \tilde L_k} \leq a_{k+1} \Big| \tilde L_1 = \ell_1, \dots, \tilde L_k = \ell_k \Bigr)
= \bbP_{n - \ell_1 - \dots - \ell_k} \Bigl( \tfrac{\tilde L_1}{n - \ell_1 - \dots - \ell_k} < a_{k+1} \Bigr).
\ee
(Notice that the $(k+1)$th cycle in the left side has become the 1st cycle in the right side.) The right side of the equation converges to the beta measure of $(0,a_{k+1})$. Convergence is uniform in $(\ell_1,\dots,\ell_k) \in A$ because of \eqref{numbers diverge}. The right side of \eqref{conditional prob} then converges to the beta measure of the product of intervals $\times_{i=1}^{k+1} (0,a_i)$ by the induction hypothesis.
\end{proof}

\section{Parameters with sub-exponential decay}
\label{sec regime E}

The second regime with long cycles occurs for parameters that go slowly to 0, such as $\theta_n = n^{-\gamma}$ with $\gamma>0$, or $\theta_n = \e{-n^\gamma}$ with $0<\gamma<1$. It is not hard to check that the assumptions of the theorem below are satisfied in both these cases. Notice that the results about the $R_{j}$s and about $K$ have already been proved in \cite{BG} in the case $\theta_{n} \sim n^{-\gamma}$.

\begin{theorem}
\label{thm small parameters}
Assume that $0< \frac{\theta_{n-j} \theta_j}{\theta_n} < c_j$ for all $n$ and all $1 \leq j \leq \frac n2$, with constants $c_j$ that satisfy $\sum_{j\geq1} \frac{c_j}j < \infty$. Assume also that $\frac{\theta_{n+1}}{\theta_n} \to 1$ as $n\to\infty$. Then $\sum_j h_j < \infty$, and
\[
\lim_{n\to\infty} \bbP_n (L_1 = n-m) = \frac{h_m}{\sum_{j\geq0} h_j}.
\]
In addition, $R_1, R_2, R_3, \dots$ converge weakly to independent Poisson random variables with respective means $\theta_1, \frac{\theta_2}2, \frac{\theta_3}3, \dots$, and $K-1$ converges to Poisson with mean $\sum_j \frac{\theta_j}j$.
\end{theorem}

\begin{proof}
The claim about $L_1$ was proved in \cite{BUV}. The claim about the $R_j$s follows from Corollary \ref{cor Poisson} and from the fact that
\be
h_n = \frac{C \theta_n}n \bigl( 1 + o(1) \bigr)
\ee
with $C = \sum h_j$. This was proved in \cite{BUV}, see Eq.\ (3.12) there.

In order to prove that $K-1$ converges to a Poisson random variable, let $m,k$ be fixed. We consider the set of permutations
\be
A = \{ \sigma : R_1(\sigma) + \dots + R_m(\sigma) = k-1 \},
\ee
and $B$ the set of permutations where exactly one cycle has length larger than $m$. We have, for all $n>2m$,
\be
A \cap B \subset \{ \sigma: K(\sigma) = k \}, \qquad \{ \sigma:  L_1(\sigma) \geq n-m \} \subset B.
\ee
Then
\be
\bbP_n(K=k) \geq \bbP_n(A) - \bbP(\{ L_1 \geq n-m \}^{\rm c}) = \bbP_n(A) - 1 + \sum_{j=0}^m \bbP_n(L_1 = n-j).
\ee
We take the limit $n\to\infty$. Since $R_1+\dots+R_m$ converges to Poisson with mean $\sum_{j=1}^m \frac{\theta_j}j$, we get
\be
\liminf_{n\to\infty} \bbP_n(K=k) \geq \frac1{(k-1)!} \Bigl( \sum_{j=1}^m \frac{\theta_j}j \Bigr)^{k-1} \e{-\sum_{j=1}^m \frac{\theta_j}j} - 1 + \frac{\sum_{j=0}^m h_j}{\sum_{j\geq0} h_j}.
\ee
We now take the limit $m\to\infty$ and we get
\be
\liminf_{n\to\infty} \bbP_n(K=k) \geq \frac1{(k-1)!} \Bigl( \sum_{j\geq1} \frac{\theta_j}j \Bigr)^{k-1} \e{-\sum_{j\geq1} \frac{\theta_j}j}.
\ee
Summing over $k\geq1$, the left side is less or equal to 1 by Fatou's lemma; the right side yields 1. This shows that the inequality above is actually an identity, and $K-1$ is indeed Poisson in the limit $n\to\infty$.
\end{proof}

\section{Parameters with super-exponential decay}
\label{sec regime F}

We conclude our study of random permutations with cycle weights by discussing the case $\theta_n = \e{-n^\gamma}$ with $\gamma>1$. It was actually studied in \cite{BUV}, where the typical cycle length was proved to be a fractional power of $\log n$, namely
\be
\frac{L_1}{((\gamma-1) \log n )^{1/\gamma}} \Rightarrow 1.
\ee
We complement this result with a claim about the number of finite cycles. It is actually not very sharp, but it provides useful information nonetheless.

\begin{theorem}
As $n\to\infty$, we have
\[
\bbE_n(R_j) = \exp\Bigl\{ j \gamma \Bigl( \frac{\log n}{\gamma-1} \Bigr)^{\frac{\gamma-1}\gamma} + o \Bigl( (\log n)^{\tfrac{\gamma-1}\gamma} \Bigr) \Bigr\}.
\]
\end{theorem}

\begin{proof}
The radius of convergence of $G_\theta$ and $G_h$ is now infinite. Let $r_n$ satisfy $I_0(r_n) = n$. It was shown in \cite{BUV}, see Eq.\ (4.32) there, that
\be
\label{asymptotic r_n superdecay}
r_n = \exp\Bigl\{ \gamma \Bigl( \frac{\log n}{\gamma-1} \Bigr)^{\frac{\gamma-1}\gamma} (1+o(1)) \Bigr\}.
\ee
It follows that
\be
r_{n-j} = \exp\Bigl\{ \gamma \Bigl( \frac{\log n}{\gamma-1} \Bigr)^{\frac{\gamma-1}\gamma} \Bigl( 1 + \frac{\log( 1- \frac jn)}{\log n} \Bigr)^{\frac{\gamma-1}\gamma} (1+o(1)) \Bigr\}.
\ee
We can then express $r_{n-j}$ in term of $r_n$,
\be
r_{n-j} = r_n \e{o \bigl( (\log n)^{\frac{\gamma-1}\gamma} \bigr)},
\ee
where the precise meaning of $o(\cdot)$ is that for any $\varepsilon>0$, there exists $N$ such that
\be
\Bigl| \frac{o \bigl( (\log n)^{\frac{\gamma-1}\gamma} \bigr)}{(\log n)^{\frac{\gamma-1}\gamma}} \Bigr| < \varepsilon
\ee
for all $j,n$ such that $n>N$ and $n-j>N$.

Next, we observe that the parameters satisfy
\be
\e{-(j-1)^\gamma} = \e{-j^\gamma + \gamma j^{\gamma-1} + O(j^{\gamma-2})}
\ee
so that $j \e{-j^\gamma} \leq \e{-(j-1)^\gamma}$ for all $j$ large enough. It follows that if $r$ is large enough,
\be
I_1(r) = \sum_{j\geq1} j \e{-j^\gamma} r^j \leq \sum_{j\geq1} \e{-(j-1)^\gamma} r^j = r (I_0(r)+1).
\ee
Since $I_1(r)$ is increasing in $r$ and $r_n$ is increasing in $n$, we have for $n$ large enough,
\be
1 \leq \frac{I_1(r_n)}{I_1(r_{n-j})} \leq r_n \frac{I_0(r_n)+1}{I_0(r_{n-j})}.
\ee
Using Proposition \ref{prop great bounds} and Eq.\ \eqref{asymptotic r_n superdecay}, we have
\be
\frac{h_{n-j}}{h_n} = \exp\Bigl\{ j \gamma \Bigl( \frac{\log n}{\gamma-1} \Bigr)^{\frac{\gamma-1}\gamma} + o \Bigl( j (\log n)^{\frac{\gamma-1}\gamma} \Bigr) \Bigr\}.
\ee
A special case of Proposition \ref{prop rvs} (c) is $\bbE_n(R_j) = \frac{h_{n-j}}{h_n} \frac{\theta_j}j$. Combining this with the previous equation, and neglecting $\theta_j/j$ which is less than the error, we get the claim of the theorem for all $j$ finite.
\end{proof}

\section{Uniform saddle point estimates}
\label{sec uniform SPE}

Since the two cases of generating functions that we have considered are entirely similar, we initially restrict attention to the case of the generating function for the $h_n$ (as specified  in (\ref{Gh})).  $G_h(z) = \exp I_{-1}(z)$ is analytic in the unit disc and hence the normalization coefficients we want to study are naturally given by the Cauchy representation
\be
\label{Cauchyrep}
h_n = \frac{1}{2\pi \ii} \oint_{\mathcal{C}_n} G_{h}(z) \frac{\dd z}{z^{n+1}}
\ee
where $\mathcal{C}_n$ is the circle centered at $0$ of radius $r_n$.  With respect to polar coordinates along $\mathcal{C}_n$ this becomes
\be
\label{Cauchyrep2}
h_n = \frac{1}{2\pi} \int_{-\pi}^\pi \exp\left( I_{-1}\left(r_n \e{\ii \phi}\right) -n \log\left(r \e{\ii \phi}\right)\right) \dd\phi.
\ee
We observe that the function 
\be
\label{exponent}
F_n(z) = I_{-1}(z) - n \log(z)
\ee
which is positive and continuous on $(0,1)$ approaches $\infty$ at both endpoints and therefore attains a minimum value at $ r_n$. This point is unique and, as we have already observed, explicitly given as the critical point satisfying
\be
0 = F_n^\prime(z) = \frac{d}{dz} I_{-1}(z) - \frac{n}{z},
\ee
or equivalently
\be
r_n = I_0^{-1}(n).
\ee

We now consider a complex neighborhood (in $z$) of $r_n$. Since $F_n(z)$ is analytic in the unit disc minus the origin, one may assume that $F_n$ is analytic on the chosen neighborhood of $r_n$ and hence $r_n$ must be a saddle point of $F_n$ (by the maximum principle). The integral (\ref{Cauchyrep2}) is complex-valued, so it is natural to try to apply the method of steepest descent \cite{Miller} here. One may describe this approach in terms of a dynamical system; viz., the Cauchy-Riemann equations for the analytic function $F_n(z)$ may be viewed as a gradient dynamical system with potential $\Re F_n(z)$. The critical point $z = r_n$ is a fixed point of this system and  the locus $\Im F(n) = 0$ cuts out the stable and unstable manifolds of this fixed point. The real axis is the stable manifold in all the cases we consider. The steepest descent curves are the components of the unstable manifold; $F_n(z) = \Re F_n(z)$ decreases monotonically along these curves as one moves away from the fixed point $r_n$. The situation is illustrated in the left graphic of Figure \ref{fig:stpdes} below which depicts these stable and unstable manifolds for a particular case of algebraically growing parameters.  For background the reader is referred to section 2.3 of \cite{Miller}. We begin by Taylor expanding $F_n(z)$, as given by (\ref{exponent}), near $r_n$, and applying Taylor's form of the remainder theorem to derive the representation
\begin{align}
F_n(z) &= \left(I_{-1}(r_n) - n\log(r_n)\right) +  \frac{I_1(r_n)}{2 r_n^2} (z - r_n)^2 + \frac{I_2(\tilde{r_n}) - 3 I_1(\tilde{r_n}) + 2 I_0(\tilde{r_n}) - n}{6 \tilde{r_n}^3} (z - r_n)^3\nn\\
&= F_n(r_n) + A_n (z - r_n)^2 + B_n  (z - r_n)^3 (1 + o(1)).
\end{align}
where 
\be
\begin{split}
\tilde{r_n} &= r_n(1 + o(1))\\
A_n &= \frac{I_1(r_n)}{2 r_n^2}\\
B_n &= \frac{I_2({r_n}) - 3 I_1({r_n}) + n}{6 {r_n}^3}.
\end{split}
\ee
With $z = x + iy$, the local structure of the stable and unstable manifolds is given by the locus
\be
\begin{split} \label{Fn}
\Im F_n(z) &= 2 A_n (x - r_n) y + B_n \left(3 (x - r_n)^2 y - y^3\right)\\
&= y \left( 2 A_n (x - r_n)  + B_n \left(3 (x - r_n)^2  - y^2\right)\right)\\
&= 0.
\end{split}
\ee 
Indeed, $y = 0$ locally describes the stable manifold which we have already seen to be the $x$-axis while the remaining factor, which to leading orders has the form 
\begin{equation}\label{unsmfld}
y^2 = \frac{2A_n}{B_n}(x - r_n),
\end{equation}
locally describes a parabolic arc for the unstable manifold (steepest descent curves), consistent with the example shown in Figure \ref{fig:stpdes}.

One may similarly expand the real part of $F_n$ (first line below) and then restrict it to the unstable manifold (second line below),
\begin{align} 
\Re F_n(z) - F_n(r_n) &= A_n \left[(x - r_n)^2 - y^2 \right]+ B_n \left[(x - r_n)^3  - 3 (x - r_n) y^2\right] (1 + o(1))\nn\\
\label{finalform} &= \left( A_n (x - r_n)^2 + B_n (x - r_n)^3 \right) (1 + o(1))
\end{align}
where in the second line we have used (\ref{unsmfld}) and the fact, which will be seen below, that $B_n$ dominates $A_n$ as $n \to \infty$. 

We next apply these observations to the contour integral (\ref{Cauchyrep}).  By Cauchy's Theorem the contour of integration, $\mathcal{C}_n$ may be deformed within a region of analyticity without affecting the value of the integral. We will deform to a contour of the form $\mathcal{C} = D_+ +  C - D_-$ where $D_+$ is the  sub-locus of the steepest descent curve in the upper half plane starting at $r_n$ and terminating at a point $z_0$, inside the unit disc, to be determined. $D_-$ is the conjugate reflection of $D_+$ in the lower half plane. $C$ is the circular arc of radius $|z_0|$ starting at $z_0$ and terminating at $\bar{z_0}$. (Note that although $\mathcal{C}_n$ is not equal to the steepest descent path, it is tangent to that path at $r_n$.)   

We concentrate first on the integral over the steepest descent contours of $\mathcal{C}$. At the end of this section it will be shown that $z_0$ may be chosen, depending on $n$, so that that the quadratic term in (\ref{finalform}) goes to infinity with $n$ while the cubic term goes to zero. $z_0$ itself will tend to $1$ along with $r_n$ as $n\to \infty$. It is then natural to make the change of variables along $D_\pm$:
\be
\begin{split}
\frac{\sigma^2}2 &=  F(r_n) - F_n(z)\\
\label{CoV} &= - \frac{I_1(r_n)}{2 r_n^2} (x - r_n)^2 + o(1).
\end{split}
\ee
With this we have
\be
\begin{split}
\frac1{2\pi \ii} \int_{D_+ - D_-} \e{F_n(z)} \frac{\dd z}{z} &= \frac{\e{F_n(r_n)}}{2\pi \ii} \int_{\bar{z_0}}^{z_0} \e{F_n(z) - F_n(r_n)} \frac{\dd z}{z}\\
&=  \frac{\e{F_n(r_n)}}{\pi \ii} \int_{x_0 - r_n}^{0} \e{-\sigma^2/2} \frac{\dd x}{r_n + x}\\
&=  \frac{\e{F_n(r_n)}}{\pi \sqrt{I_1(r_n)}} \int_0^{\sigma_0} \e{-\sigma^2/2} \dd\sigma (1 + o(1))
\end{split}
\ee
where $\sigma_0 = \frac{\sqrt{I_1(r_n)}}{r_n} (x_0 - r_n)$. In the last line the change of variables (\ref{CoV}) was implemented. As already mentioned, at the end of this section it will be shown that a choice of $z_0$ can be made consistent with all prior estimates and for which $\sigma_0 \to \infty$ as $n \to \infty$. It follows that
\begin{eqnarray} \label{main}
\frac1{2\pi \ii} \int_{D_+ - D_-} \e{F_n(z)} \frac{\dd z}{z} &=& \frac{\e{F_n(r_n)}}{\sqrt{2 \pi I_1(r_n)}} (1 + o(1)). 
\end{eqnarray}

To complete the verification of (\ref{asymptotic h}), as well as the similar argument for  (\ref{general saddle asymptotic}), one still needs to argue that the global error coming from the integral (\ref{Cauchyrep2}),  restricted to $C$, is asymptotically negligible in comparison to (\ref{main}).  We illustrate the situation with two images from the case of algebraic growth (specifically, the instance of $G_h$ for (\ref{algparams}) where $\gamma = 1$ with $n = 100$). 

\begin{figure} [h]  
   \includegraphics[width=2.5in]{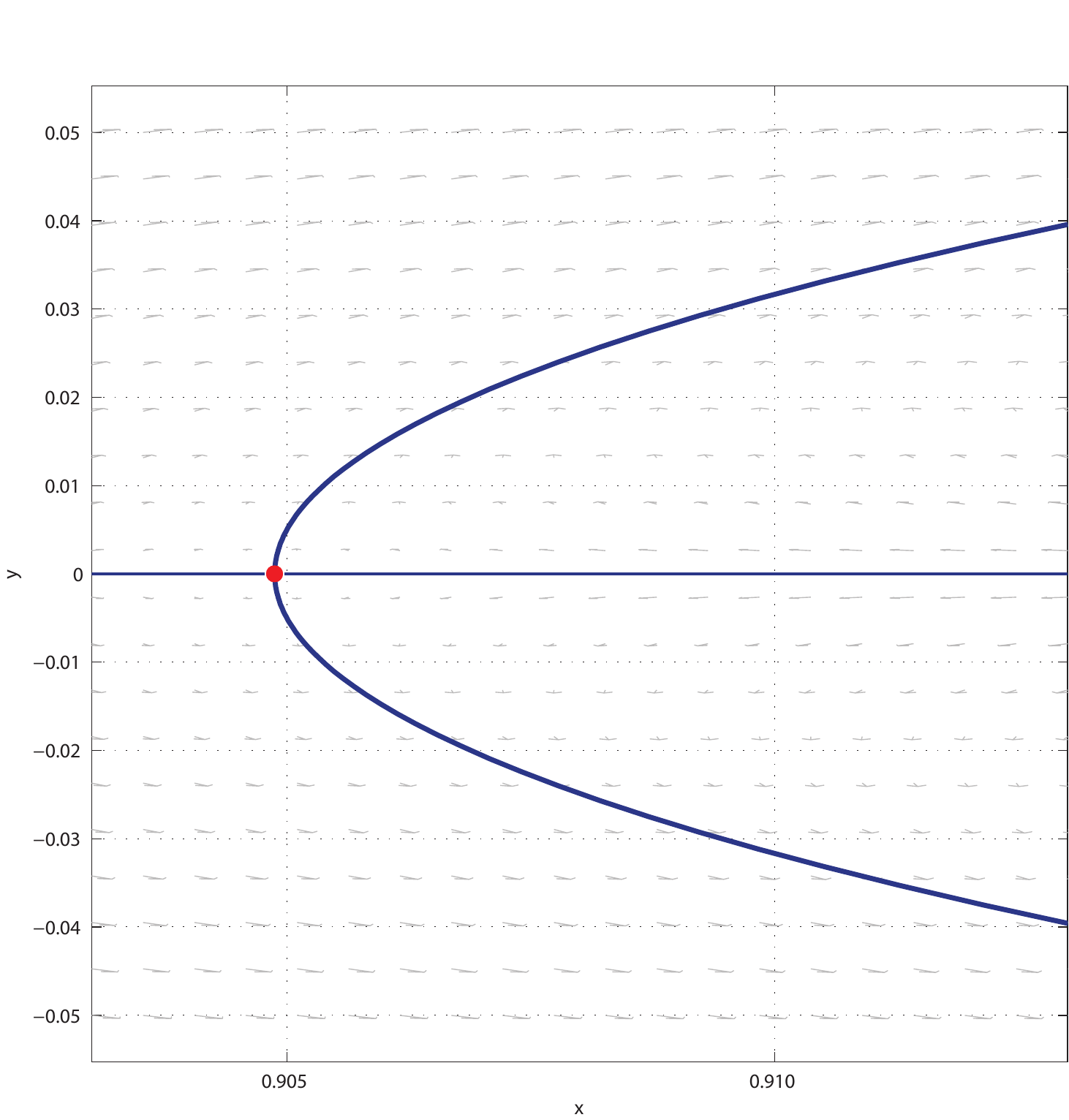} 
   \includegraphics[width=2.5in]{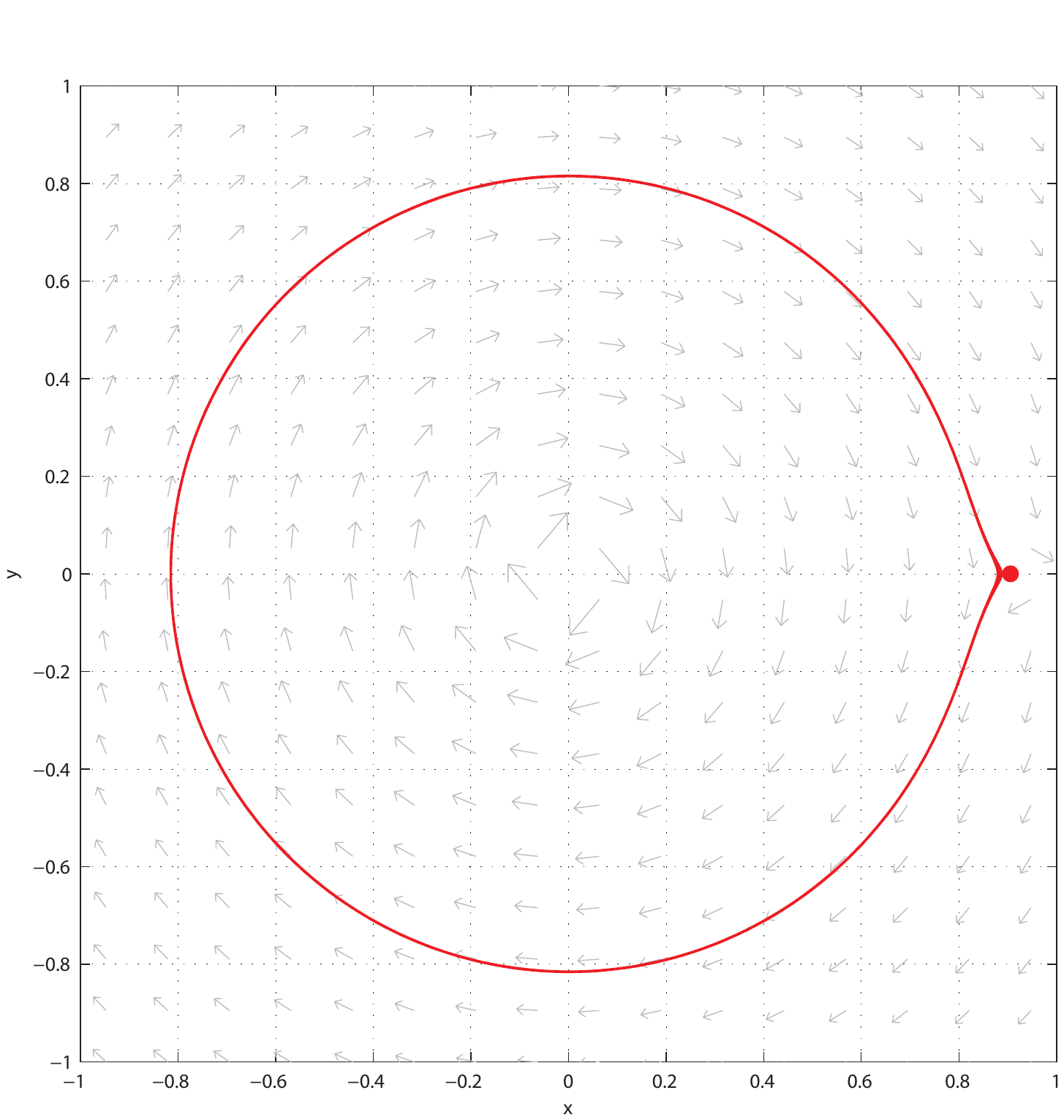}
   \caption{ saddle point and level curve at $r_n$}
   \label{fig:stpdes}
   \end{figure}

As has already been described, the graphic on the left in Figure \ref{fig:stpdes} shows the saddle point $r_n$ with paths on the real axis ascending from the saddle and the other two curves descending from the saddle ({\it steepest descent curves}) into the upper and lower half planes respectively. This illustrates the fact, stated before, that the contour $\mathcal{C}_n$ is tangent to the steepest descent curves. The graphic on the right shows the level curve, passing through $r_n$, of the real part of $F_n(z)$.  This level curve is also a locus where the magnitude of the integrand of (\ref{Cauchyrep2}) is constant. Note that this level curve is quite close to being circular away from a small neighborhood of $r_n$. This property is shared by other nearby level curves (for level values different than $F_n(r_n)$). This suggests that the order of the magnitude of the global error is bounded by the order of the absolute value of the integrand of (\ref{Cauchyrep2}) evaluated at $r_n \e{ i \phi_0}$. We will take a slightly different tack here which essentially accomplishes the same estimate but is easier to implement. Namely, we return to $C$ and observe that the value of $F(z)$ (which equals $\Re F(z)$ along the unstable manifolds) is decreasing along the unstable manifolds as one moves away from $r_n$. Hence, the  value of the integrand in (\ref{Cauchyrep}) at the respective endpoints $z_0, \bar{z_0} = |z_0|e^{i\phi_0}$ of $C$ is exponentially smaller (in $n$) than its value at $r_n$. We further observe that
\be
\begin{split}
\left| \frac{1}{2\pi \ii}\int_{C} \frac{\e{n  I_{-1}(z)}}{z^{n+1}} \dd z\right| 
& \leq  \frac{1}{\pi}\frac{1}{|z_0|^n}\int_{\phi_0}^\pi \e{n \Re I_{-1}(|z_0| \e{\ii\phi})} \dd\phi\\
&=  \frac{1}{\pi}\frac{1}{|z_0|^n} \e{n \Re I_{-1}(|z_0| \e{\ii\phi_0})} \int_{\phi_0}^\pi \e{n \Re \left[ I_{-1}(|z_0| \e{\ii\phi}) - 
I_{-1}|z_0| \e{\ii\phi_0})\right]} \dd\phi\\
& \leq \frac{1}{\pi}\frac{1}{|z_0|^n} \e{n \Re I_{-1}(|z_0| \e{\ii\phi_0})} (\pi - \phi_0),
\end{split}
\ee
where the last inequality follows from the fact that
\be
\Re \left[ I_{-1}(|z_0| \e{\ii\phi}) -  I_{-1}(|z_0| \e{\ii\phi_0})\right] = \sum_{j \geq 1} \frac{\theta_j}{j} |z_0|^j \left( \cos(j\phi) - \cos(j\phi_0)\right) \leq 0 
\ee
for $\phi \in (\phi_0, \pi)$. It follows that the integral over $C$ is exponentially negligible in comparison to (\ref{main}). 
\smallskip

Finally we return to the claim made just prior to (\ref{CoV}) that $z_0$ may be chosen so that, in (\ref{finalform}), the quadratic term grows to infinity with $n$ while the cubic term decreases. 
Given the analysis presented in sections 4 and 5 and in particular the estimates (\ref{the right choice}), (\ref{subexpscale}) and (\ref{algscale}, \ref{algI1}) it suffices to show that the order of 
$x_0 - r_n$ may be chosen so that
\begin{eqnarray*}
I_1(r_n) (x_0 - r_n)^2 &\to& \infty\\
I_2(r_n) (x_0 - r_n)^3 &\to& 0
\end{eqnarray*}
as $n \to \infty$.  
The following table summarizes the orders in $n$ of the relevant terms and presents a choice for the orders of  $(x_0 - r_n)$ for each of the cases of generating functions that we consider in this paper. In each case the choice is given in terms of a weighted geometric mean of the growth rates for $I_1(r_n)$ and $I_2(r_n)$. (Note that in the case of $G_\theta$, $r_n$ should be replaced by $\rho_n$.) 

{\small
\begin{table}[h]
\begin{center}
\begin{tabular}{cccccc}
&$\mathcal{O}( I_1(r_n))$ &$\mathcal{O}( I_2(r_n))$ & $ \mathcal{O}(x_0 - r_n)$ & $\mathcal{O}( I_1(r_n)(x_0 \! - \! r_n)^2)$ & $\mathcal{O}( I_2(r_n)(x_0 \! - \! r_n)^3)$\\
\hline 
$G_h(z)$, alg. & $n^{\frac{\gamma + 2}{\gamma + 1}}$ & $n^{\frac{\gamma + 3}{\gamma + 1}}$ &$n^{-\frac1{12}\frac{5\gamma + 12}{\gamma + 1}}$ & $n^{\frac16\frac{\gamma}{\gamma + 1}}$ & $n^{-\frac14\frac{\gamma}{\gamma + 1}}$\\
\hline
$G_\theta(z)$, sub-exp. & $n^{2-\gamma}$ & $n^{3 - 2\gamma}$ & $n^{- 1 + \frac7{12} \gamma}$ & $n^{\frac16 \gamma}$ & $n^{-\frac14 \gamma}$\\
\hline
$G_h(z)$, sub-exp. & $n (\log n)^{\frac1{\gamma}}$ & $n (\log n)^{\frac2 \gamma}$ & $n^{-\frac5{12}} (\log n)^{-\frac7{12} \frac1\gamma}$ & $\Bigl(\frac{n}{(\log n)^{1/\gamma}}\Bigr)^{\frac16}$ & $ \Bigl(\frac{(\log n)^{\frac1\gamma}}{n}\Bigr)^{\frac14}$\\
\hline
\end{tabular}
\end{center}
\end{table}
} 

\medskip

The arguments in this section follow the general strategy of {\it Hayman's method}. We refer the reader to Chapter VIII of \cite{FS} for a nice overview of this technique.

\medskip
\noindent {\footnotesize
{\bf Funding:} 
This work was supported by the National Science Foundation [DMS-0808059 to N.M.E.];  and the Engineering and Physical Sciences Research Council [EP/G056390/1 to D.U.].

\end{document}